\documentclass[12pt]{article}

\usepackage{amssymb,amsthm,amsmath}
\usepackage[dvips]{graphicx}
\usepackage[dvips]{color}
\usepackage{times}
\usepackage{enumerate}

\newcommand{\real}{\mathop{\mathrm{Re}}}
\newcommand{\supp}{\mathop{\mathrm{supp}}}

\newtheorem{theorem}{Theorem}[section]
\newtheorem{lemma}[theorem]{Lemma}
\newtheorem{remark}[theorem]{Remark}
\newtheorem{proposition}[theorem]{Proposition}

\numberwithin{equation}{section}

\title{Increasing stability in an inverse problem for the acoustic equation}
\author{Sei Nagayasu\thanks{Department of Mathematical Science,
Graduate School of Material Science,
University of Hyogo, 2167 Shosha, Himeji, Hyogo 671-2280, Japan.
Email:sei@sci.u-hyogo.ac.jp} \qquad Gunther
Uhlmann\thanks{Department of Mathematics, University of Washington,
Box 354305, Seattle, WA 98195-4350 and Department of Mathematics,
University of California, Irvine, CA 92697-3875, USA.
Email:gunther@math.washington.edu}\qquad Jenn-Nan
Wang\thanks{Department of Mathematics, NCTS (Taipei), National
Taiwan University, Taipei 106, Taiwan.
Email:jnwang@math.ntu.edu.tw}}

\date{}

\begin{document}
\maketitle

\begin{abstract}
In this work we study the inverse boundary value problem of
determining the refractive index in the acoustic equation. It is
known that this inverse problem is ill-posed. Nonetheless,  we
show that the ill-posedness decreases when we increase the frequency
and the stability estimate changes from logarithmic type for low
frequencies to a Lipschitz estimate for large frequencies.
\end{abstract}

\section{Introduction}\label{sec1}
\setcounter{equation}{0}

In this paper we study the issue of stability for determining the
refractive index in the acoustic equation by boundary measurements.
It is well known that this inverse problem is ill-posed. However,
one anticipates that the stability will increase if one increases
the frequency. This phenomenon was observed numerically in the
inverse obstacle scattering problem \cite{CHP}. Several rigorous
justifications of the increasing stability phenomena in different
settings were obtained by Isakov {\it et al} \cite{HI, I07, I11,
ASI07, ASI10}. Especially, in \cite{I11}, Isakov considered the
Helmholtz equation with a potential
\begin{equation}\label{eq0}
-\Delta u-k^2u+qu=0\quad\text{in}\quad\Omega.
\end{equation}
He obtained stability estimates of determining $q$ by the
Dirichlet-to-Neumann map for different ranges of $k$'s. All of these results
demonstrate the increasing stability phenomena in $k$.
For the case of the inverse source problem for Helmholtz equation and an homogeneous background it was shown in \cite{BLT} that the ill-posedness of the inverse problem decreases as the frequency increases.

In this paper, we study the acoustic wave equation. Let $\Omega
\subset \mathbb{R}^{n}$ be a bounded domain, where $n\ge 3$. Let
$\partial \Omega$ be smooth. We consider the equation
\begin{equation}\label{eq:eq}
\bigl( \Delta + k^{2} q(x) \bigr) u(x) = 0\quad \mbox{ in }\quad
\Omega,
\end{equation}
where the real-valued $q(x)$ is the refractive index.
Assume that the
kernel of the operator $\Delta+k^2 q(x)$ on $H_0^1(\Omega)$ is
trivial. Associated with \eqref{eq:eq}, we define the
Dirichlet-to-Neumann map (DN map) $\Lambda:
H^{1/2}(\partial\Omega)\to H^{-1/2}(\partial\Omega)$ by
\[
\Lambda f =
\frac{\partial u}{\partial \nu} \biggr|_{\partial \Omega} ,
\]
where $u$ is the solution to (\ref{eq:eq}) with the Dirichlet
condition $u = f$ on $\partial \Omega$, and $\nu$ is the unit outer
normal vector of $\partial \Omega$. The uniqueness of this inverse
problem is well known \cite{SU}. This inverse problem is notoriously
ill-posed. For this aspect, Alessandrini proved that the stability
estimate for this problem is of log type \cite{A88} and Mandache showed
that the log type stability is optimal \cite{Ma}. In this paper, we
would like to focus on how the stability behaves when the frequency
$k$ increases. Now we state the main result.
\begin{theorem}\label{thm}
Assume that $q_1(x)$ and $q_2(x)$ are two sound speeds with
associated DN maps $\Lambda_1$ and $\Lambda_2$, respectively. Let $s
> (n/2) + 1$, $M > 0$. Suppose $\lVert q_{l} \rVert_{H^{s}
( \Omega )} \leq M$ $(l = 1, 2)$ and $\supp ( q_{1} - q_{2} )
\subset \Omega$. Denote $\widetilde{q}$ a zero extension of $q_{1} -
q_{2}$. Then there exists a constant $C_1$, depending only on $n$,
$s$, and $\Omega$, such that if $k^{2} \geq 1/(C_1M)$ and $\lVert
\Lambda_{1} - \Lambda_{2} \rVert_{\ast} \leq 1 / e$ then
\begin{equation}\label{iest}
\lVert \widetilde{q} \rVert_{H^{-s} ( \mathbb{R}^{n} )}
\leq \frac{C}{k^{2}} \exp ( C k^{2} )
\lVert \Lambda_{1} - \Lambda_{2} \rVert_{\ast}
+ C \left(
 k^{2}
 + \log \frac{1}{\lVert \Lambda_{1} - \Lambda_{2} \rVert_{\ast}}
\right)^{- ( 2 s - n )}
\end{equation}
holds, where $C>0$ depends only on $n, s, \Omega , M$
and $\supp ( q_{1} - q_{2} )$.
Here $\lVert \cdot \rVert_{\ast}$ is the operator norm
from $H^{1/2}(\partial\Omega)$ into $H^{-1/2}(\partial\Omega)$.
\end{theorem}
\begin{remark}\label{rem1}

\smallskip\noindent {\rm 1.} The estimate \eqref{iest} consists two
parts -- Lipschitz and logarithmic estimates. As $k$ increases, the
logarithmic part decreases and the Lipschitz part becomes dominated.
In other words, the ill-posedness is alleviated when $k$ is large.

\medskip\noindent {\rm 2.} We would like to remark on the constant
$C\exp(Ck^2)/k^2$ appearing in the Lipschitz part of \eqref{iest}.
$1/k^2$ comes from $k^2q$ in the equation, which appears naturally,
while, $\exp(Ck^2)$ is due to the fact that we use the complex
geometrical optics solutions in the proof. Even so, we expect that the
there is an exponential growth of the constant with frequency since we do not assume
any geometrical restriction on $q(x)$ other than regularity. For the wave equation it has been shown by Burq for the obstacle problem {\rm\cite{B}} that the local energy decay is log-slow and this is due to the presence of trapped rays. Notice that in our case we can have trapped rays. For the case of simple sound speeds we expect that there is no exponential increase in the constant. In {\rm\cite{StU}} a H\"older stability estimate was obtained for the hyperbolic DN map for generic simple metrics. For very general metrics there is not known modulus of continuity for the hyperbolic DN map, see {\rm\cite{AKKLT}} for convergence results.

However, in practice, $k$ is fixed and so is
the constant. Therefore, one should expect to obtain a better
resolution of $q$ from boundary measurements when the chosen $k$ is
large.

\medskip\noindent {\rm 3.} Unlike the result in {\rm\cite[Theorem~2.1]{I11}} {\rm (for equation \eqref{eq0})} where
the stability estimates were derived in different ranges of $k$,
estimate \eqref{iest} is valid for all range of $k$ provided
$k^2\geq 1/(C_1 M)$ .
\end{remark}

The proof of Theorem~\ref{thm} makes use of Alessandrini's arguments
\cite{A88} and the CGO solutions constructed in \cite{SU}. The main task
is to keep track of how $k$ appears in the proof of the stability estimates.

\section{Complex geometrical optics solutions}
\setcounter{equation}{0}

In this section, we construct CGO solutions to the equation
(\ref{eq:eq}) by using the idea in \cite{SU}. The main point is to
express the dependence of constants on $k$ explicitly. We first
state two easy consequences from the results in \cite{SU}.

\begin{lemma}[{see \cite[Proposition~2.1 and Corollary~2.2]{SU}}]%
\label{lemma:SUprop21alpha}
Let $s \geq 0$ be an integer.
Let $\varepsilon_{0} > 0$.
Let $\xi \in \mathbb{C}^{n}$
satisfy
$\xi \cdot \xi = 0$ and $\lvert \xi \rvert \geq \varepsilon_{0}$.
Then for any $f \in H^{s} ( \Omega )$
there exists $w \in H^{s} ( \Omega )$ such that
$w$ is a solution to
\[
\Delta w + \xi \cdot \nabla w = f \mbox{ in } \Omega
\]
and satisfies the estimate
\[
\lVert w \rVert_{H^{s} ( \Omega )}
\leq \frac{C_{0}}{\lvert \xi \rvert}
\lVert f \rVert_{H^{s} ( \Omega )} ,
\]
where a positive constant $C_{0}$
depends only on $n, s, \varepsilon_{0}$ and $\Omega$.
\end{lemma}

By using this lemma, we can obtain a solution to the equation
\begin{equation}\label{eq:SUtheorem23alphaeq}
\Delta \psi + \xi \cdot \nabla \psi + g \psi = f
\end{equation}
satisfying some decaying property as in the following lemma.
\begin{lemma}[{\cite[Theorem~2.3 and Corollary~2.4]{SU}}]%
\label{lemma:SUtheorem23alpha}
Let $s > n/2$ be an integer.
Let $\varepsilon_{0} > 0$.
Let $\xi \in \mathbb{C}^{n}$ satisfy
$\xi \cdot \xi = 0$ and $\lvert \xi \rvert \geq \varepsilon_{0}$.
Let $f, g \in H^{s} ( \Omega )$.
Then there exists $C_{1} > 0$ depending only on
$n, s, \varepsilon_{0}$ and $\Omega$
such that if
\[
\lvert \xi \rvert
\geq C_{1} \lVert g \rVert_{H^{s} ( \Omega )}
\]
then there exists a solution $\psi \in H^{s} ( \Omega )$
to the equation
{\rm (\ref{eq:SUtheorem23alphaeq})}
satisfying the estimate
\[
\lVert \psi \rVert_{H^{s} ( \Omega )}
\leq \frac{2 C_{0}}{\lvert \xi \rvert}
\lVert f \rVert_{H^{s} ( \Omega )} ,
\]
where $C_{0}$ is the positive constant in
Lemma~\ref{lemma:SUprop21alpha}.
\end{lemma}

The needed CGO solutions are constructed as follows.
\begin{proposition}\label{prop:CGO}
Let $s > n/2$ be an integer.
Let $\varepsilon_{0} > 0$.
Let $\xi \in \mathbb{C}^{n}$ satisfy
$\xi \cdot \xi = 0$ and $\lvert \xi \rvert \geq \varepsilon_{0}$.
Define the constants $C_{0}$ and $C_{1}$ as in
Lemma~\ref{lemma:SUtheorem23alpha}.
Then if
\[
\lvert \xi \rvert \geq C_{1} k^{2} \lVert q \rVert_{H^{s} ( \Omega
)}
\]
then there exists a solution $u$ to the equation {\rm (\ref{eq:eq})}
with the form of
\begin{equation}\label{eq:CGOform}
u(x) = \exp \left( \frac{\xi}{2} \cdot x \right)
\bigl( 1 + \psi (x) \bigr) ,
\end{equation}
where $\psi$ has the estimate
\[
\lVert \psi \rVert_{H^{s} ( \Omega )} \leq \frac{2 C_{0}
k^{2}}{\lvert \xi \rvert} \lVert q \rVert_{H^{s} ( \Omega )} .
\]
\end{proposition}
\begin{proof}
Substituting (\ref{eq:CGOform}) into (\ref{eq:eq}),
we have
\[
\Delta \psi + \xi \cdot \nabla \psi + k^{2} q \psi = - k^{2} q.
\]
Then by Lemma~\ref{lemma:SUtheorem23alpha},
we obtain this proposition.
\end{proof}

\section{Proof of stability estimate}
\setcounter{equation}{0}

This section is devoted to the proof of Theorem~\ref{thm}. We begin
with Alessandrini's identity.
\begin{proposition}\label{prop:identity}
Let $u_{l}$ be a solution to {\rm (\ref{eq:eq})} with $q = q_{l}$,
then we have
\[
k^{2} \int_{\Omega}
 ( q_{2} - q_{1} ) u_{1} u_{2} \,
d x = \bigl\langle
 ( \Lambda_{1} - \Lambda_{2} ) u_{1} |_{\partial \Omega} , \,
 u_{2} |_{\partial \Omega}
\bigr\rangle.
\]
\end{proposition}

Now we would like to estimate the Fourier transform of the
difference of two $q$'s. We denote $\mathcal{F}(f)$ the Fourier
transformation of a function $f$.
\begin{lemma}\label{lemma:Fourierest}
Let $s > (n/2) + 1$ be an integer and $M > 0$. Assume $\lVert q_{l}
\rVert_{H^{s} ( \Omega )} \leq M$, $\supp ( q_{1} - q_{2} ) \subset
\Omega$ and $k^{2} \geq 1 / C_{1} M$, where  $C_{1}$ is the constant
defined in Lemma~\ref{lemma:SUtheorem23alpha} corresponding to
$\varepsilon_{0} = 1$. Let $\widetilde{q}$ be a zero extension of
$q_{1} - q_{2}$ and $a_0\ge C_1$. Suppose that $\chi \in
C_{0}^{\infty} ( \Omega )$ satisfies $\chi \equiv 1$ near $\supp (
q_{1} - q_{2} )$. Then for $r \geq 0$ and $\eta \in \mathbb{R}^{n}$
with $\lvert \eta \rvert = 1$ the following statements hold: if $0
\leq r \leq a_{0} k^{2} M$ then
\begin{equation}\label{eq:lowerFourierestimate}
\lvert \mathcal{F} \widetilde{q} ( r \eta ) \rvert \leq \frac{C
\lVert \chi \rVert_{H^{s} ( \Omega )}}{a_{0}} \lVert \widetilde{q}
\rVert_{H^{-s} ( \mathbb{R}^{n} )} + \frac{C}{k^{2}} \exp ( C a_{0}
k^{2} M ) \lVert \Lambda_{1} - \Lambda_{2} \rVert_{\ast}
\end{equation}
holds; if $r \geq C_{1} k^{2} M$ then
\begin{equation}\label{eq:higherFourierestimate}
\lvert \mathcal{F} \widetilde{q} ( r \eta ) \rvert \leq \frac{C M
k^{2} \lVert \chi \rVert_{H^{s} ( \Omega )}}{r} \lVert \widetilde{q}
\rVert_{H^{-s} ( \mathbb{R}^{n} )} + \frac{C}{k^{2}} \exp ( C r )
\lVert \Lambda_{1} - \Lambda_{2} \rVert_{\ast}
\end{equation}
holds, where $C>0$ depends only on $n, s$ and $\Omega$.
\end{lemma}
\begin{proof}
In the following proof, the letter $C$ stands for a general constant
depending only on $n, s$ and $\Omega$. By
Proposition~\ref{prop:CGO}, we can construct CGO solutions $u_{l}
(x)$ to the equation (\ref{eq:eq}) with $q = q_{l}$ having the form
of
\[
u_{l} (x) = \exp \left( \frac{\xi_{l}}{2} \cdot x \right)
\bigl( 1 + \psi_{l} (x) \bigr)
\]
for $l = 1, 2$,
and we have
\begin{align}
& \int_{\Omega}
 ( q_{2} - q_{1} )
 \exp \left( \frac{1}{2} ( \xi_{1} + \xi_{2} ) \cdot x \right)
 ( 1 + \psi_{1} + \psi_{2} + \psi_{1} \psi_{2} ) \,
d x \notag \\
& = \frac{1}{k^{2}} \bigl\langle
 ( \Lambda_{1} - \Lambda_{2} ) u_{1} |_{\partial \Omega} , \,
 u_{2} |_{\partial \Omega}
\bigr\rangle \label{eq:identityCGO}
\end{align}
from Proposition~\ref{prop:identity}, where $\psi_{l}$ satisfies
\[
\lVert \psi_{l} \rVert_{H^{s} ( \Omega )} \leq \frac{C k^{2}}{\lvert
\xi_{l} \rvert} \lVert q_{l} \rVert_{H^{s} ( \Omega )}
\]
if $\xi_{l} \in \mathbb{C}^{n}$ satisfies
$\xi_{l} \cdot \xi_{l} = 0$,
$\lvert \xi_{l} \rvert \geq 1$ and
\begin{equation}\label{eq:xil}
\lvert \xi_{l} \rvert \geq C_{1} k^{2} \lVert q_{l} \rVert_{H^{s} (
\Omega )} .
\end{equation}
We remark that
$\lVert \psi_{l} \rVert_{H^{s} ( \Omega )} \leq C$
also holds. Indeed, we have
\[
\lVert \psi_{l} \rVert_{H^{s} ( \Omega )} \leq \frac{C k^{2} \lVert
q_{l} \rVert_{H^{s} ( \Omega )}}
          {\lvert \xi_{l} \rvert}
\leq \frac{C k^{2} \lVert q_{l} \rVert_{H^{s} ( \Omega )}}
          {C_{1} k^{2} \lVert q_{l} \rVert_{H^{s} ( \Omega )}}
= \frac{C}{C_{1}} = C.
\]

Now, let $r \geq 0$, and $\eta \in \mathbb{R}^{n}$ satisfy $\lvert
\eta \rvert = 1$. We assume that $\alpha, \zeta \in \mathbb{R}^{n}$
satisfy
\begin{equation}\label{eq:kzeta}
\alpha \cdot \eta = \alpha \cdot \zeta = \eta \cdot \zeta = 0 \mbox{
and } \lvert \zeta \rvert^{2} = \lvert \alpha \rvert^{2} + r^{2} .
\end{equation}
Define $\xi_{1}$ and $\xi_{2}$ as
\[
\xi_{1} = \zeta + i \alpha - i r \eta\quad \mbox{ and }\quad \xi_{2}
= - \zeta - i \alpha - i r \eta .
\]
Then we have
\[
\xi_{l} \cdot \xi_{l} = 0, \
\lvert \xi_{l} \rvert^{2} = \lvert \zeta
\rvert^{2} + \lvert \alpha \rvert^{2} + r^{2} = 2 \lvert \zeta
\rvert^{2} ~ ( l = 1, 2 ) \mbox{ and } \frac{1}{2} ( \xi_{1} +
\xi_{2} ) = - i r \eta .
\]
Hence by (\ref{eq:identityCGO}), we immediately obtain that
\begin{align}
\mathcal{F} \widetilde{q} ( r \eta )
& = - \int_{\Omega}
 ( q_{2} - q_{1} ) \exp ( - i r \eta \cdot x )
 ( \psi_{1} + \psi_{2} + \psi_{1} \psi_{2} ) \,
d x \notag \\
& \hspace*{3ex} + \frac{1}{k^{2}} \bigl\langle
 ( \Lambda_{1} - \Lambda_{2} ) u_{1} |_{\partial \Omega} , \,
 u_{2} |_{\partial \Omega}
\bigr\rangle \label{eq:Fourierqtilde1}
\end{align}
provided $\lvert \xi_{l} \rvert \geq 1$ and (\ref{eq:xil}) are
satisfied. We first estimate the first term on the right hand side
of \eqref{eq:Fourierqtilde1} by
\begin{align*}
& \left\lvert
 \int_{\Omega}
  ( q_{2} - q_{1} ) \exp ( - i r \eta \cdot x )
  ( \psi_{1} + \psi_{2} + \psi_{1} \psi_{2} ) \,
 d x
\right\rvert \\
& = \left\lvert
 \int_{\Omega}
  ( q_{2} - q_{1} ) \exp ( - i r \eta \cdot x )
  \chi ( \psi_{1} + \psi_{2} + \psi_{1} \psi_{2} ) \,
 d x
\right\rvert \\
& \leq \lVert q_{2} - q_{1} \rVert_{H^{-s} ( \Omega )}
\bigl\lVert
 \chi ( \psi_{1} + \psi_{2} + \psi_{1} \psi_{2} )
\bigr\rVert_{H^{s} ( \Omega )}\\
& \leq \lVert \widetilde{q} \rVert_{H^{-s} ( \mathbb{R}^{n} )}
\lVert \chi \rVert_{H^{s} ( \Omega )}
\bigl(
 \lVert \psi_{1} \rVert_{H^{s} ( \Omega )}
 + \lVert \psi_{2} \rVert_{H^{s} ( \Omega )}
 + \lVert \psi_{1} \rVert_{H^{s} ( \Omega )}
 \lVert \psi_{2} \rVert_{H^{s} ( \Omega )}
\bigr) \\
& \leq \lVert \widetilde{q} \rVert_{H^{-s} ( \mathbb{R}^{n} )}
\lVert \chi \rVert_{H^{s} ( \Omega )}
\left(
 \frac{C k^{2}}{\sqrt{2} \lvert \zeta \rvert}
 + \frac{C k^{2}}{\sqrt{2} \lvert \zeta \rvert}
 + C \frac{C k^{2}}{\sqrt{2} \lvert \zeta \rvert}
\right)
\sum_{l = 1}^{2} \lVert q_{l} \rVert_{H^{s} ( \Omega )} \\
& = \frac{C k^{2} \lVert \chi \rVert_{H^{s} ( \Omega )}}
         {\lvert \zeta \rvert}
\lVert \widetilde{q} \rVert_{H^{-s} ( \mathbb{R}^{n} )}
\sum_{l = 1}^{2} \lVert q_{l} \rVert_{H^{s} ( \Omega )} .
\end{align*}
since
\begin{math}
\chi ( \psi_{1} + \psi_{2} + \psi_{1} \psi_{2} )
\in H_{0}^{s} ( \Omega )
\end{math}
and $s > n/2$.

On the other hand, by taking $R$ large enough such that $\Omega
\subset B_{R} (0)$, we have
\begin{align*}
\bigl\lVert
 u_{l} |_{\partial \Omega}
\bigr\rVert_{L^{2} ( \partial \Omega )}
& \leq \lvert \partial \Omega \rvert^{1/2}
\lVert u_{l} \rVert_{C^{0} ( \Omega )}
\leq \lvert \partial \Omega \rvert^{1/2}
\exp \left( \frac{\lvert \real \xi_{l} \rvert}{2} R \right)
\bigl( 1 + \lVert \psi_{l} \rVert_{L^{\infty} ( \Omega )} \bigr) \\
& \leq C
\exp \left( \frac{\lvert \real \xi_{l} \rvert}{2} R \right)
\bigl( 1 + \lVert \psi_{l} \rVert_{H^{s} ( \Omega )} \bigr) \\
& \leq C \exp \left( \frac{\lvert \real \xi_{l} \rvert}{2} R \right)
( 1 + C ) = C \exp \left(
\frac{\lvert \zeta \rvert}{2} R \right).
\end{align*}
Likewise, we can get that
\begin{align*}
\bigl\lVert
 \nabla u_{l} |_{\partial \Omega}
\bigr\rVert_{L^{2} ( \partial \Omega )}
& = \left\lVert
 \frac{\xi_{l}}{2} u_{l}
 + \exp \left( \frac{\xi_{l}}{2} \cdot \bullet \right)
 ( \nabla \psi_{l} )
\right\rVert_{L^{2} ( \partial \Omega )} \\
& \leq \frac{\sqrt{2} \lvert \zeta \rvert}{2}
C \exp \left( \frac{\lvert \zeta \rvert}{2} R \right)
+ \lvert \partial \Omega \rvert^{1/2}
\exp \left( \frac{\lvert \zeta \rvert}{2} R \right)
\lVert \nabla \psi_{l} \rVert_{C^{0} ( \Omega )} \\
& \leq C \lvert \zeta \rvert
\exp \left( \frac{\lvert \zeta \rvert}{2} R \right)
+ C \exp \left( \frac{\lvert \zeta \rvert}{2} R \right)
\lVert \nabla \psi_{l} \rVert_{H^{s-1} ( \Omega )} \\
& \leq C \lvert \zeta \rvert
\exp \left( \frac{\lvert \zeta \rvert}{2} R \right)
+ C \exp \left( \frac{\lvert \zeta \rvert}{2} R \right)
\lVert \psi_{l} \rVert_{H^{s} ( \Omega )} \\
& \leq C \exp ( C \lvert \zeta \rvert )
\end{align*}
since $s -1 > n/2$. Consequently, we have
\[
\bigl\lVert
 u_{l} |_{\partial \Omega}
\bigr\rVert_{H^{1/2} ( \partial \Omega )}
\leq C \exp ( C \lvert \zeta \rvert ).
\]
Therefore, we can estimate the second term of the right-hand side of
(\ref{eq:Fourierqtilde1}) by
\begin{align*}
\left\lvert \bigl\langle
 ( \Lambda_{1} - \Lambda_{2} ) u_{1} |_{\partial \Omega} , \,
 u_{2} |_{\partial \Omega}
\bigr\rangle \right\rvert
& \leq \lVert \Lambda_{1} - \Lambda_{2} \rVert_{\ast}
\bigl\lVert
 u_{1} |_{\partial \Omega}
\bigr\rVert_{H^{1/2} ( \partial \Omega )}
\bigl\lVert
 u_{2} |_{\partial \Omega}
\bigr\rVert_{H^{1/2} ( \partial \Omega )} \\
& \leq C \exp ( C \lvert \zeta \lvert ) \lVert \Lambda_{1} -
\Lambda_{2} \rVert_{\ast} .
\end{align*}
Summing up, we have shown that for $r > 0$ and for $\eta \in
\mathbb{R}^{n}$ with $\lvert \eta \rvert = 1$ if we take $\alpha$
and $\zeta$ satisfying the conditions (\ref{eq:kzeta}), $\lvert
\zeta \rvert \geq 2^{-1/2}$ and
\begin{equation}\label{eq:zeta}
\lvert \zeta \rvert \geq 2^{-1/2} C_{1} k^{2} \lVert q_{l}
\rVert_{H^{s} ( \Omega )}
\end{equation}
then
\begin{align}
\lvert \mathcal{F} \widetilde{q} ( r \eta ) \rvert & \leq \frac{C
k^{2} \lVert \chi \rVert_{H^{s} ( \Omega )}}
          {\lvert \zeta \rvert}
\lVert \widetilde{q} \rVert_{H^{-s} ( \mathbb{R}^{n} )}
\sum_{l=1}^{2} \lVert q_{l} \rVert_{H^{s} ( \Omega )} \notag \\
& \hspace*{3ex} \mbox{} + \frac{C}{k^{2}} \exp ( C \lvert \zeta
\rvert ) \lVert \Lambda_{1} - \Lambda_{2} \rVert_{\ast}
\label{eq:Fouriersummingup}
\end{align}
holds.

Now assume that $\lVert q_{l} \rVert_{H^{s} ( \Omega )} \leq M$ and
$k^{2} \geq 1 / C_{1} M$. Thus if
\begin{equation}\label{eq:zetaM}
\lvert \zeta \rvert \geq C_{1} k^{2} M
\end{equation}
holds, then (\ref{eq:zeta}) and $\lvert \zeta \rvert \geq 2^{-1/2}$
are satisfied. Pick $a_{0} \geq C_{1}$. We first consider the case
where $0 \leq r \leq a_{0} k^{2} M$. By choosing $\alpha$ and
$\zeta$ satisfying
\[
\alpha \cdot \eta = \alpha \cdot \zeta = \eta \cdot \zeta = 0, \
\lvert \zeta \rvert = a_{0} k^{2} M ( \geq r ) \mbox{ and } \lvert
\alpha \rvert = \sqrt{
 ( a_{0} k^{2} M )^{2} - r^{2}
}
\]
both (\ref{eq:kzeta}) and (\ref{eq:zetaM}) are then satisfied since
$a_{0} \geq C_{1}$. Hence we obtain (\ref{eq:Fouriersummingup}),
that is (\ref{eq:lowerFourierestimate}). On the other hand, when $r
\geq C_{1} k^{2} M$, we can choose $\alpha = 0$, $\eta \cdot \zeta =
0$ and $\lvert \zeta \rvert = r$. Then (\ref{eq:kzeta}),
(\ref{eq:zetaM}) are satisfied and thus (\ref{eq:Fouriersummingup})
holds and consequently (\ref{eq:higherFourierestimate}) is valid.
\end{proof}

Now we prove our main result.

\begin{proof}
As above, $C$ denotes a general constant depending only on $n, s$
and $\Omega$. Written in polar coordinates, we have
\begin{align}\label{i123}
\lVert \widetilde{q} \rVert_{H^{-s} ( \mathbb{R}^{n} )}^{2}
& = C \int_{0}^{\infty} \int_{\lvert \eta \rvert = 1}
 \lvert \mathcal{F} \widetilde{q} ( r \eta ) \rvert^{2}
 ( 1 + r^{2} )^{-s} r^{n-1} \,
d \eta \, d r \notag\\
& = C \biggl(
 \int_{0}^{a_{0} k^{2} M} \int_{\lvert \eta \rvert = 1}
  \lvert \mathcal{F} \widetilde{q} ( r \eta ) \rvert^{2}
  ( 1 + r^{2} )^{-s} r^{n-1} \,
 d \eta \, d r \notag\\
 & \hspace*{7ex} \mbox{}
 + \int_{a_{0} k^{2} M}^{T} \int_{\lvert \eta \rvert = 1}
  \lvert \mathcal{F} \widetilde{q} ( r \eta ) \rvert^{2}
  ( 1 + r^{2} )^{-s} r^{n-1} \,
 d \eta \, d r \notag\\
 & \hspace*{7ex} \mbox{}
 + \int_{T}^{\infty} \int_{\lvert \eta \rvert = 1}
  \lvert \mathcal{F} \widetilde{q} ( r \eta ) \rvert^{2}
  ( 1 + r^{2} )^{-s} r^{n-1} \,
 d \eta \, d r
\biggr) \notag\\
& =: C( I_{1} + I_{2} + I_{3} ),
\end{align}
where $a_{0} \geq C_{1}$ and $T \geq a_{0} k^{2} M$ are parameters
which will be chosen later. Here $C_1$ is the constant given in
Lemma~\ref{lemma:Fourierest}. From now on, we take $k^2\ge
1/(C_1M)$.

Our task now is to estimate each integral separately. We begin with
$I_{3}$. Since
\begin{math}
\lvert \mathcal{F} \widetilde{q} ( r \eta ) \rvert
\leq C \lVert q_{1} - q_{2} \rVert_{L^{2} ( \Omega )}
\end{math},
$q_{1} - q_{2} \in H_{0}^{s} ( \Omega )$, and $s > n/2$, we have
that
\begin{align}\label{i3}
I_{3}
& \leq C \int_{T}^{\infty}
 \lVert q_{1} - q_{2} \rVert_{L^{2} ( \Omega )}^{2}
 ( 1 + r^{2} )^{-s} r^{n-1} \,
d r
\leq C T^{- m}
\lVert q_{1} - q_{2} \rVert_{L^{2} ( \Omega )}^{2} \notag\\
& \leq C T^{- m}
\left(
 \varepsilon \lVert q_{1} - q_{2} \rVert_{H^{-s} ( \Omega )}^{2}
 + \frac{C}{\varepsilon}
 \lVert q_{1} - q_{2} \rVert_{H^{s} ( \Omega )}^{2}
\right) \notag\\
& \leq C T^{- m}
\left(
 \varepsilon \lVert \widetilde{q} \rVert_{H^{-s} ( \mathbb{R}^{n} )}^{2}
 + \frac{M^{2}}{\varepsilon}
\right)
\end{align}
for $\varepsilon > 0$, where $m := 2 s - n$.

On the other hand, by Lemma~\ref{lemma:Fourierest}, we can estimate
\begin{align}\label{i1}
I_{1} & \leq C \int_{0}^{a_{0} k^{2} M}
 ( 1 + r^{2} )^{-s} r^{n-1} \,
d r \notag\\
& \hspace*{5ex} \mbox{} \times
\left[
 \frac{\lVert \chi \rVert_{H^{s} ( \Omega )}^{2}}{a_{0}^{2}}
 \lVert \widetilde{q} \rVert_{H^{-s} ( \mathbb{R}^{n} )}^{2}
 + \frac{\exp ( 2 C a_{0} k^{2} M )}{k^{4}}
 \lVert \Lambda_{1} - \Lambda_{2} \rVert_{\ast}^{2}
\right] \notag\\
& \leq C \int_{0}^{\infty}
 ( 1 + r^{2} )^{-s} r^{n-1} \,
d r
\left[
 \frac{C_{\chi}^{2}}{a_{0}^{2}}
 \lVert \widetilde{q} \rVert_{H^{-s} ( \mathbb{R}^{n} )}^{2}
 + \frac{\exp ( C a_{0} k^{2} M )}{k^{4}}
 \lVert \Lambda_{1} - \Lambda_{2} \rVert_{\ast}^{2}
\right] \notag\\
& = \frac{C C_{\chi}^{2}}{a_{0}^{2}} \lVert \widetilde{q}
\rVert_{H^{-s} ( \mathbb{R}^{n} )}^{2} + \frac{C \exp ( C a_{0}
k^{2} M )}{k^{4}} \lVert \Lambda_{1} - \Lambda_{2} \rVert_{\ast}^{2}
,
\end{align}
where $\chi \in C_{0}^{\infty} ( \Omega )$ satisfies $\chi \equiv 1$
near $\supp ( q_{2} - q_{1} )$ and $C_{\chi} := \lVert \chi
\rVert_{H^{s} ( \Omega )}$. In view of
\begin{align*}
\int_{a_{0} k^{2} M}^{T}
 ( 1 + r^{2} )^{-s} r^{n-3} \,
d r & \leq \int_{a_{0} k^{2} M}^{T}
 r^{- 2 s + n-3} \,
d r
\leq C ( a_{0} k^{2} M )^{- 2 s + n - 2} \\
& \leq C ( a_{0} k^{2} M )^{-2} ( C_{1} k^{2} M )^{- m} \leq
\frac{C}{a_{0}^{2} k^{4} M^{2}}
\end{align*}
and
\begin{align*}
\int_{a_{0} k^{2} M}^{T}
 \exp ( C r ) ( 1 + r^{2} )^{-s} r^{n-1} \,
d r & \leq \exp ( C T ) \int_{a_{0} k^{2} M}^{T}
 ( 1 + r^{2} )^{-s} r^{n-1} \,
d r \\
& \leq \exp ( C T ) \int_{0}^{\infty}
 ( 1 + r^{2} )^{-s} r^{n-1} \,
d r \\
& \leq C \exp ( C T ),
\end{align*}
we have that
\begin{align}\label{i2}
I_{2} & \leq C M^{2} k^{4} \lVert \chi \rVert_{H^{s} ( \Omega )}^{2}
\lVert \widetilde{q} \rVert_{H^{-s} ( \mathbb{R}^{n} )}^{2}
\int_{a_{0} k^{2} M}^{T}
 ( 1 + r^{2} )^{-s} r^{n-3} \,
d r \notag\\
& \hspace*{5ex} \mbox{} + \frac{C}{k^{4}} \lVert \Lambda_{1} -
\Lambda_{2} \rVert_{\ast}^{2} \int_{a_{0} k^{2} M}^{T}
 \exp ( C r ) ( 1 + r^{2} )^{-s} r^{n-1} \,
d r \notag\\
& \leq \frac{C C_{\chi}^{2}}{a_{0}^{2}} \lVert \widetilde{q}
\rVert_{H^{-s} ( \mathbb{R}^{n} )}^{2} + \frac{C}{k^{4}} \exp ( C
T ) \lVert \Lambda_{1} - \Lambda_{2} \rVert_{\ast}^{2}.
\end{align}
Combining \eqref{i123}--\eqref{i2} gives
\begin{align*}
\lVert \widetilde{q} \rVert_{H^{-s} ( \mathbb{R}^{n} )}^{2}
& \leq C ( I_{1} + I_{2} + I_{3} ) \\
& \leq \frac{C C_{\chi}^{2}}{a_{0}^{2}} \lVert \widetilde{q}
\rVert_{H^{-s} ( \mathbb{R}^{n} )}^{2} + \frac{C \exp ( C a_{0}
k^{2} M )}{k^{4}}
\lVert \Lambda_{1} - \Lambda_{2} \rVert_{\ast}^{2} \\
& \hspace*{5ex} \mbox{} + \frac{C C_{\chi}^{2}}{a_{0}^{2}} \lVert
\widetilde{q} \rVert_{H^{-s} ( \mathbb{R}^{n} )}^{2} +
\frac{C}{k^{4}} \exp ( C T )
\lVert \Lambda_{1} - \Lambda_{2} \rVert_{\ast}^{2} \\
& \hspace*{5ex} \mbox{} + C T^{- m}
\left(
 \varepsilon \lVert \widetilde{q} \rVert_{H^{-s} ( \mathbb{R}^{n} )}^{2}
 + \frac{M^{2}}{\varepsilon}
\right) \\
& = \left(
 \frac{C_{2}^{2} C_{\chi}^{2}}{a_{0}^{2}}
 + C_{3} T^{- m} \varepsilon
\right) \lVert \widetilde{q} \rVert_{H^{-s} ( \mathbb{R}^{n} )}^{2} \\
& \hspace*{5ex} \mbox{} + \frac{C}{k^{4}} \bigl(
 \exp ( C a_{0} k^{2} M ) + \exp ( C T )
\bigr) \lVert \Lambda_{1} - \Lambda_{2} \rVert_{\ast}^{2} + \frac{C
M^{2}}{\varepsilon} T^{- m} ,
\end{align*}
where positive constants $C_{2}$ and $C_{3}$ depend only on $n, s$
and $\Omega$.

Now we pick $a_{0}$ and $\varepsilon$ as
\[
a_{0} =
2 C_{2} C_{\chi} \geq C_{1}
\mbox{ and }
\varepsilon = \frac{T^{m}}{4 C_{3}}
\]
(if needed, we take $C_{2}$ large enough). We then obtain that
\begin{align}
\lVert \widetilde{q} \rVert_{H^{-s} ( \mathbb{R}^{n} )}^{2}
& \leq \frac{C}{k^{4}} \left[
 \exp ( 2 C_{2} C C_{\chi} k^{2} M ) + \exp ( C T )
\right] \lVert \Lambda_{1} - \Lambda_{2} \rVert_{\ast}^{2}
+ C T^{- 2 m} M^{2} \notag \\
& = \frac{C}{k^{4}} \exp ( C a k^{2} ) A + C \Phi (T)
\label{eq:keyestimatewithT}
\end{align}
for
\begin{math}
T \geq a_{0} k^{2} M = 2 C_{2} C_{\chi} k^{2} M = a k^{2}
\end{math},
where
\[
\Phi (T) := \frac{1}{k^{4}} \exp ( C_{4} T ) A
+ M^{2} T^{- 2 m} ,
\]
$A := \lVert \Lambda_{1} - \Lambda_{2} \rVert_{\ast}^{2}$, $a := 2
C_{2} C_{\chi} M^{2}$ and $C_{4} > 0$ depends only on $n, s$ and
$\Omega$.

To continue, we consider two cases:
\begin{equation}\label{eq:Lambdasmall}
a k^{2} \leq p \log \frac{1}{A}
\end{equation}
and
\begin{equation}\label{eq:Lambdanotsmall}
a k^{2} \geq p \log \frac{1}{A} ,
\end{equation}
where $p$ will be determined later (see (\ref{eq:p})).

For the first case (\ref{eq:Lambdasmall}), our aim is to show that
there exists $T \geq a k^{2}$ such that
\begin{equation}\label{eq:Lambdasmallaim}
\Phi (T) \leq 2 C_{5} \left( k^{2} + \log \frac{1}{A} \right)^{- 2 m} .
\end{equation}
Substituting (\ref{eq:Lambdasmallaim}) into
(\ref{eq:keyestimatewithT}) clearly implies (\ref{iest}). Now to
derive (\ref{eq:Lambdasmallaim}), it is enough to prove that
\begin{equation}\label{eq:Lambdasmallaim1}
\frac{1}{k^{4}} \exp ( C_{4} T ) A
\leq C_{5} \left( k^{2} + \log \frac{1}{A} \right)^{- 2 m}
\end{equation}
and
\begin{equation}\label{eq:Lambdasmallaim2}
M^{2} T^{- 2 m}
\leq C_{5} \left( k^{2} + \log \frac{1}{A} \right)^{- 2 m} .
\end{equation}
Remark that (\ref{eq:Lambdasmallaim2}) in equivalent to
\[
T \geq C_{5}^{- 1 / 2 m} M^{1 / m}
\left( k^{2} + \log \frac{1}{A} \right) ,
\]
which holds if
\begin{equation}\label{eq:Lambdasmallaim2b}
T \geq C_{5}^{- 1 / 2 m} M^{1/m}
\left( 1 + \frac{p}{a} \right) \log \frac{1}{A}
\end{equation}
because of (\ref{eq:Lambdasmall}). Setting $T = p \log (1/A)$ ($\geq
a k^{2}$ by (\ref{eq:Lambdasmall})), then
(\ref{eq:Lambdasmallaim2b}) holds provided
\begin{equation}\label{eq:Lambdasmallaim2C5}
p \geq C_{5}^{- 1 / 2 m} M^{1/m} \left( 1 + \frac{p}{a} \right) .
\end{equation}
Now we turn to (\ref{eq:Lambdasmallaim1}).
It is clear that (\ref{eq:Lambdasmallaim1}) is equivalent to
\begin{equation}\label{eq:Lambdasmallaim1a}
C_{4} p \log \frac{1}{A}
\leq \log C_{5} + 2 \log k^{2} + \log \frac{1}{A}
- 2 m \log \left( k^{2} + \log \frac{1}{A} \right)
\end{equation}
since $T = p \log ( 1 / A )$.
It follows from (\ref{eq:Lambdasmall}) that
\[
\log \left( k^{2} + \log \frac{1}{A} \right)
\leq \log \left(
 \frac{p}{a} \log \frac{1}{A} + \log \frac{1}{A}
\right)
= \log \left( \frac{p}{a} + 1 \right)
+ \log \log \frac{1}{A} .
\]
Hence (\ref{eq:Lambdasmallaim1a}) is verified if we can show that
\[
C_{4} p \log \frac{1}{A}
\leq \log C_{5} - 2 \log ( M C_{1} )
+ \log \frac{1}{A}
- 2 m \left( \log \left( \frac{p}{a} + 1 \right)
 + \log \log \frac{1}{A}
\right) ,
\]
i.e.
\begin{equation}\label{eq:Lambdasmallaim1b}
( 1 - C_{4} p ) \log \frac{1}{A}
- 2 m \log \log \frac{1}{A}
+ \log C_{5}
- 2 \log ( M C_{1} )
- 2 m \log \left( \frac{p}{a} + 1 \right)
\geq 0
\end{equation}
for $\log ( 1 / A ) \geq 1$. Now we choose
\begin{equation}\label{eq:p}
p = \frac{1}{2 C_{4}} .
\end{equation}
Then (\ref{eq:Lambdasmallaim1b}) becomes
\begin{equation}\label{eq:Lambdasmallaim1c}
\log \frac{1}{A} - 4 m \log \log \frac{1}{A}
+ 2 \log C_{5} - 4 \log ( M C_{1} )
- 4 m \log \left( \frac{p}{a} + 1 \right) \geq 0.
\end{equation}
Notice that
\begin{align*}
\inf_{0 < A \leq 1 / e} \left(
 \log \frac{1}{A} - 4 m \log \log \frac{1}{A}
\right)
& = \inf_{z \geq 1} ( z - 4 m \log z ) \\
& \geq \inf_{z > 0} ( z - 4 m \log z )
= 4 m \log \frac{e}{4 m} .
\end{align*}
Hence if we choose $C_{5}$ such that
\begin{equation}\label{eq:Lambdasmallaim1C5}
C_{5} \geq ( M C_{1} )^{2}
\left( \frac{p}{a} + 1 \right)^{2 m}
\left( \frac{4 m}{e} \right)^{2 m}
\end{equation}
then (\ref{eq:Lambdasmallaim1c}) follows.
Finally, we take
\[
C_{5} := \max \left\{
 C_{1}^{2} \left( \frac{4 m}{e} \right)^{2 m} , \
 p^{- 2 m}
\right\} M^{2} \left( 1 + \frac{p}{a} \right)^{2 m} ,
\]
which depends only on $n, \Omega , s, M$ and $\chi$.
With such choice of $C_{5}$,
the conditions
(\ref{eq:Lambdasmallaim1C5}) and (\ref{eq:Lambdasmallaim2C5}) hold,
and thus estimate (\ref{eq:Lambdasmallaim}) is satisfied.

Next we consider the second case (\ref{eq:Lambdanotsmall}).
By (\ref{eq:keyestimatewithT})
with $T = a k^{2}$, we get that
\begin{align*}
\lVert \widetilde{q} \rVert_{H^{-s} ( \mathbb{R}^{n} )}^{2}
& \leq \frac{C}{k^{4}} \exp ( C a k^{2} ) A
+ \frac{C}{k^{4}} \exp ( C_{4} a k^{2} ) A
+ C M^{2} ( a k^{2} )^{- 2 m} \\
& \leq \frac{C}{k^{4}} \exp ( C a k^{2} ) A
+ C M^{2} a^{- 2 m} k^{- 4 m} .
\end{align*}
Hence it remains to show that
\[
k^{- 4 m} \leq C_{6} \left( k^{2} + \log \frac{1}{A} \right)^{- 2 m} ,
\]
i.e.
\begin{equation}\label{eq:Lambdanotsmallaim}
k^{2} \geq C_{6}^{- 1 / 2 m}
\left( k^{2} + \log \frac{1}{A} \right) .
\end{equation}
Since
\[
k^{2} + \log \frac{1}{A}
\leq \left( 1 + \frac{a}{p} \right) k^{2}
\]
by (\ref{eq:Lambdanotsmall}), we have (\ref{eq:Lambdanotsmallaim})
if we take $C_{6}$ large enough so that
\[
C_{6} \geq \left( 1 + \frac{a}{p} \right)^{2 m} .
\]
The proof is completed.
\end{proof}

\section*{Acknowledgements}
Nagayasu was partially supported
by Grant-in-Aid for Young Scientists (B). Uhlmann was partly supported by NSF
and a Visiting Distinguished Rothschild Fellowship at the Isaac Newton Institute.
Wang was partially supported by the National Science Council of
Taiwan. We would also like to thank P. Stefanov for helpful discussions.

\end{document}